\documentclass[12pt,letterpaper]{amsart}
\makeatletter
\@namedef{subjclassname@2020}{%
\textup{2020} Mathematics Subject Classification}
\makeatother
\usepackage{geometry,xcolor}
\usepackage{setspace}
\usepackage{mathpazo} 
\usepackage[colorlinks,allcolors=blue]{hyperref} 
\usepackage{xpatch,multirow}
\xpatchcmd{\proof}{\itshape}{\bfseries}{}{}
\newtheorem{theorem}{Theorem}
\newtheorem{proposition}{Proposition}
\newtheorem{corollary}{Corollary}
\newtheorem{lemma}{Lemma}

\theoremstyle{remark}
\newtheorem{remark}{Remark}

\usepackage{soul}
\allowdisplaybreaks

\title{On the negatively pinched properties of  the disc bundles over negatively pinched K\"ahler manifolds}
\author{Yihong Hao}
\address[Yihong Hao]{Department of Mathematics, Northwest University, Xi'an \rm{710127}, China}
\email{haoyihong@126.com}

\author{Mingming Chen}
\address[Mingming Chen]{School of Mathematics and Statistics, Henan Normal University, Xinxiang \rm{453007}, China}
\email{chenmingming105@126.com}

\author{An Wang}
\address[An Wang]{School of Mathematical Sciences, Capital Normal University, Beijing \rm{100048}, China}
\email{wangan@cnu.edu.cn}

\subjclass[2020]{32Q05, 32Q20,  32L05, 32Q02}
\keywords{negative curvature manifold, K\"ahler-Einstein manifold, holomorphic vector bundle, Hartogs domain}
\date{\today}
\begin{document}
	
\begin{abstract}
In this paper, we study the relation between the existence of a negatively (holomorphically) pinched  K\"ahler metric on a complex manifold $M$
and its disc bundle contained in a Hermitian line bundle over $M$.
\end{abstract}
\maketitle

\section{Introduction}
For a K\"ahler manifold $(M, g_{M})$, it  is called negatively $\delta$-pinched if there exist two positive real numbers, $A$ and $\delta$, such that
 $$-A\leq  \text{sectional \ curvature}  \leq -A\delta,$$ where  $0<\delta\leq 1$.
 It is called negatively $\delta$-holomorphically pinched (or $\delta$-bisectional pinched or $\delta$-Ricci pinched) if the sectional curvature is replaced by  holomorphic sectional curvature (or bisectional curvature or Ricci curvature). The constant $A$ in the inequalities is not essential, since we can always  normalize the metric by scaling.

The definition shows that a negatively $1$-pinched manifold is isometric to a real Hyperbolic space and negatively $1$-holomorphically pinched manifold is holomorphically isometric to a complex Hyperbolic space $\mathbb{CH}^{n}$ equipped with its standard metric.
A result proved independently by Hernandez \cite{Hernandez1991} and Yau and Zheng \cite{Yau1991} states that, if a compact K\"ahler
manifold $M$ is endowed with a  metric g that is negatively
$\frac{1}{4}$-pinched, then $(M, g)$ is isometric to a quotient of $\mathbb{CH}^n$.
On the other hand,
there also exist  some complex manifolds which does not admit a complete K\"ahler metric with negatively pinched  (holomorphic) sectional
curvature. Seshadri's result told us that a product of complex manifolds  cannot admit a complete K\"ahler metric with
sectional curvature $\kappa < c < 0$ and Ricci curvature $Ric > d$, where $c$ and $d$ are
constants \cite{Seshadri2006}. This implies that product domains in $\mathbb{C}^{n}$ do not admit complete
K\"ahler metrics with negatively pinched  sectional curvature.
Seshadri and Zheng \cite{Seshadri2008} also proved that the product of two complex manifolds
does not admit any complete K\"ahler metric whose bisectional curvature is pinched
between by two negative constants. For more detailed information about the topic, the readers are referred to the earlier  articles \cite{Mok1987,Seo2012,Shi1989,Yang1976,Zheng1994}.

From the work of Gromov and Thurston in \cite{Gromov1991}, one knows that there are
many  negatively $\delta$-pinched Riemannian manifolds.
However, up to 1992, there are few  examples on the complete negatively $\delta$-pinched  K\"ahler manifolds. At that  time, maybe all  known examples  had been listed by Cheung and Wu in \cite{Cheung1992}.
Along this line, methods such as direct computation,
relying on the inherent holomorphic symmetries, the deformations of the unit ball or ellipsoidal domains, the intersection of two complex ellipsoidal domains  arose.
In recently, Bakkacha \cite{BAKKACHA2024} provided  a new method to give more  complete K\"ahler manifolds with
negative sectional curvature. Actually, he   proved that a bounded domain in $\mathbb{C}^{n}$
admitting a complete K\"ahler
metric with negatively pinched holomorphic (bi)sectional curvature near
the boundary, admits a complete K\"ahler metric with negatively pinched
holomorphic (bi)sectional curvature everywhere. Hence, one can obtain  a complete K\"ahler
metric with negatively pinched sectional curvature by using the relation \eqref{ERelation}.

Due to the importance  of complete K\"ahler metrics with
negative curvature in  geometry,
we hope to provide a standard method to obtained a negatively pinched  K\"aher manifold from another one.
We primarily employ Calabi ansatz to study line bundles over  K\"ahler manifolds.

Let  $\pi:(L, h)\rightarrow M$ be a positive Hermitian line bundle over
 an $m$-dimensional  K\"ahler manifold $(M, g_{M})$ such that the K\"ahler form $\omega_{M}=-\sqrt{-1}\partial \bar\partial\log h$.
 Let $(L^{*}, h^{-1})\rightarrow M$ be the dual bundle of $L$.
 The disc bundle is defined by
 \begin{equation}\label{D1}
 D(L^{*}) := \{v \in L^{*} : |v|_{h^{-1}} < 1\},
 \end{equation}
where $|v|_{h^{-1}}$ is the norm of $v$ with respect to the metric $h^{-1}$, and we denote it by $x$ for brevity.
Let $u$ be a smooth real-valued  function on $[0,+\infty)$.
Then the following  $(1, 1)$-form
\begin{equation}\label{womega1}
\omega_{D}:=\pi^{*}(\omega_M)+\sqrt{-1}\partial\bar\partial u(|v|_{h^{-1}}^{2})
\end{equation}
is well defined on $L^*$. It is called Calabi ansatz.
By Lemma 1 in \cite{Calabi1979},
We know that
$
\omega_{D}
$
induces a K\"ahler metric in some neighbourhood of $M$  if and only  if  $u'(x)>0$ and $(xu(x)')'>0$ in $[0, 1)$.
In particular, we take $u=-\log(1-x)$. Thus we get a positive $(1,1)$-form
\begin{equation}\label{omega D1}
\omega_{D}:=\pi^{*}(\omega_M)-\sqrt{-1}\partial\bar\partial \log(1-|v|_{h^{-1}}^{2}).
\end{equation}
The respective metric is denoted by $g_{D}$.
Our main results are as follows.

\begin{theorem}
The K\"ahler manifold  $(D(L^{*}), g_{D})$ is negatively holomorphically pinched  if and only if the base manifold $(M, g_{M})$  is so.
Let $-A$ be the lower bound for the holomorphic sectional curvature of $g_{M}$. Denote by $\delta$ and $\delta'$ the pinched constants of $(M, g_{M})$ and $(D(L), g_{D})$ respectively.  Then $\delta'\geq \delta$ when $A\geq2$, and $\delta'< \delta$ when $0< A <2$.
\end{theorem}

\begin{theorem}
The K\"ahler manifold  $(D(L^{*}), g_{D})$  is negatively  pinched  if and only if $(M, g_{M})$  is.
Moreover, $\delta'\geq \frac{1}{4}\delta$ when $A\geq\frac{1}{2}$, and $\delta'< \frac{1}{4}\delta$ when $0< A <\frac{1}{2}$.
\end{theorem}

As we know, Bergman metric $g_B$, Carath\'eodory metric $g_C$, Kobayashi-Royden
metric $g_K$ and K\"ahler-Einstein metric $g_{KE}$ with negative scalar curvature are four classical invariant metrics in complex geometry.
It was proved by Wu and Yau \cite{Wu2020}
that any simply-connected complete negatively pinched K\"ahler manifold $(M, g_{M})$
has a complete Bergman metric $g_B$ and  $g_{M}$ is uniformly equivalent to
$g_B$. They also prove that any  complete negatively holomorphically pinched K\"ahler manifold $(M, g_{M})$
has a complete K\"ahler-Einstein metric $g_{KE}$ and the background K\"ahler metric $g_{M}$ is uniformly equivalent to  $g_{KE}$ and $g_K$.
It is easily to know that the holomorphic sectional curvature is dominated
by the sectional curvature \cite{Zheng2000}.  As a  result, $g_B$, $g_K$, $g_{KE}$ on
simply-connected complete negatively pinched K\"ahler manifold $(M, g_{M})$ are equivalent.
Thus, we have several corollaries directly.
\begin{corollary}
The disc bundle $(D(L^{*}), g_{D})$ over a complete negatively holomorphically pinched K\"ahler manifold  $(M, g_{M})$ has
a unique  complete K\"ahler-Einstein metric.
Moreover, the Kobayashi-Roydan metric and K\"ahler-Einstein metric are equivalent.
\end{corollary}

\begin{corollary}
Let $(M, g_{M})$ be a complete negatively pinched K\"ahler manifold.
If the disc bundle $(D(L^*), g_{D})$ over $(M, g_{M})$  is  simple-connected,  then there exists a complete  Bergman metric on $D(L^*)$.
Moreover, $g_{B}$, $g_{KE}$, $g_{K}$ and $g_{D}$ are all equivalent.
\end{corollary}

Conversely,
suppose that $A > 0$ and $M$ is  a  negatively $\delta$-holomorphically pinched  K\"ahler
manifold with  $\delta > \frac{2}{3}$, i.e. the holomorphic sectional curvature $\Theta$  satisfies $-A \leq \Theta \leq - \delta A$. Then all sectional curvatures satisfy
\begin{equation}\label{ERelation}
 -A \leq \kappa \leq -\frac{1}{4}(3\delta - 2)A<0,
\end{equation}
(see \cite{Kim2011}, \cite{Berger1967} or
\cite{Kobayashi1963} vol. II, note 23, p. 369.)
This induces that, for $\delta > \frac{2}{3}$,
 $g_{B}$, $g_{K}$, $g_{KE}$ on
simply-connected complete  $\delta$-negatively holomorphically pinched K\"ahler manifold $(M, \omega)$  are equivalent.

\begin{corollary}
 Let $D(L^*)$ be a  simple-connected disc bundle over negatively $\delta$-holomorphically pinched K\"ahler manifold $(M, \omega)$.
  If $\delta > \frac{2}{3}$, then there exists  a complete  Bergman metric and a unique complete K\"ahler-Einstein metric on $D(L^*)$.
Moreover, $g_{B}$, $g_{K}$, $g_{KE}$ are all equivalent.
\end{corollary}
\section{The geometry of the disc bundle}
In this section, we will study the completeness, the Ricci curvature and the (holomorphic) sectional curvature of the K\"ahler manifold $(D(L^{*}), g_{D})$ in \eqref{D1}.
\begin{lemma}
On the disc bundle $D(L^{*})$ over $M$, we have that
the metric $g_{D}\geq\pi^{*}(g_M)$.
\end{lemma}
\begin{proof}
Since the boundary definite function $|v|^{2}_{h^{-1}}-1=\frac{|v|^{2}}{h(z)}-1$ is negative and strictly plurisubharmonic,
 we know the $(1,1)$ form $-\sqrt{-1}\partial\bar\partial \log(1-|v|_{h^{-1}}^{2})\geq 0$ on $D(L^{*})\setminus M$.
Thus we have $g_{D}\geq\pi^{*}(g_M)$.
By \eqref{omega D1}, we have
 \begin{equation}\label{m1}
\begin{aligned}
\omega_{D}&=\frac{\sqrt{-1}}{(h-|v|^{2})^{2}}\left(h dv\wedge d\bar{v}-\bar{v} dv\wedge \bar{\partial} h
 -v\partial h\wedge d\bar{v}\right.  +\left.\partial h\wedge \bar{\partial}h-(h-|v|^{2})\partial\bar{\partial}h\right).
\end{aligned}
\end{equation}
Restrict it on $M$, $\omega_{D}=\frac{\sqrt{1}}{h}dv \wedge d\bar{v}+ \sqrt{-1}\partial \bar{\partial}(-\log h)=\frac{\sqrt{1}}{h}dv \wedge d\bar{v}+\pi^{*}(\omega_M).$ This implies that $g_{D}\geq\pi^{*}(g_M)$.
\end{proof}

\begin{lemma}
Let $D(L^{*})$ be the disc bundle  over $M$. If $g_{M}$ is complete, then
$g_{D}$ is complete.
\end{lemma}
\begin{proof}
To prove the completeness, it suffices to show that, for a fixed point $p_{0}\in M$, given any sequence
$\{x_j \}_{j=1}^{+\infty}$
of points approaching $b \in \partial D(L^*)$, $d(p_{0}, x_j )$ must diverge to $\infty$ as $j\rightarrow +\infty$.
Let
$x \in D(L^*)$ be any point and $\gamma : [0, 1] \rightarrow D(L^*)$ be a piecewise $C^{1}$-curve joining a point $p_{0}$ in $M$ to $x$. Then
$\pi\circ \gamma: [0, 1] \rightarrow M$ is a piecewise $C^{1}$-curve joining $p_{0}$ to $\pi(x) \in M$. Denote by $d_{M}(\cdot, \cdot)$,
resp. $d_{D}(\cdot, \cdot)$, the distance function induced by  the metric $g_M$ on $(M, g_{M})$, resp. by $g_D$ on $(D(L^*), g_{D})$.
Let $\{x_j \}_{j=1}^{+\infty}$
be a discrete sequence on $D(L^*)$ converging to $b \in \partial M \subset \partial D(L^*)$.
Since $(M, g_{M})$ is complete and $g_{D}\geq\pi^{*}(g_M)$, there exists a positive constant $c$ such taht  $d_{\omega_{D}}
(x_j , p_{0})$ $\geq c \cdot $ $d_{M}(\pi(x_{j} ), p_{0}) \rightarrow +\infty$.
On the other hand, if $b \in \partial D(L^*)\setminus\partial M$,
then $b$ is a smooth strictly pseudoconvex boundary point of $D(L^*)$.
In local coordinate, $$g_{D}=\sum\Psi_{j\bar{k}}dz_{j}\otimes d\bar z_{k},$$
where $\Psi=-\log(h(z)-|v|^{2})$ is the definite function of the smooth strictly pseudoconvex boundary  of $D(L^*)$.
By the discussion of Cheng-Yau in \cite{Cheng1980} (see page 509), we know that $\|\nabla \Psi\|_{g_{D}}<1$.
\begin{eqnarray*}
\lim_{s\rightarrow+\infty}\int_{0}^{s}\|\gamma'(t)\|_{g_{D}}dt
&\geq&\lim_{s\rightarrow+\infty}\int_{0}^{s}\|\nabla \Psi\|_{g_{D}}\|\gamma'(t)\|_{g_{D}}dt
\geq\lim_{s\rightarrow+\infty}\int_{0}^{s}\langle\nabla \Psi, \gamma'(t)\rangle_{g_{D}} dt\\
&=&\lim_{s\rightarrow+\infty}\int_{0}^{s} \frac{d}{dt} (\Psi\circ\gamma(t)) dt
=\lim_{s\rightarrow+\infty}(\Psi\circ\gamma(s)-\Psi\circ\gamma(0))
=+\infty.
\end{eqnarray*}
The last equation depends on the exhaustion  of $\Psi$.
\end{proof}

\begin{lemma}\label{ric}
The  Ricci curvature tensor of $g_{D}$ is
$$\mathrm{Ric}(g_D)=-(m+2)g_{D}+(m+1)g_{M}+\mathrm{Ric}(g_{M}).$$
\end{lemma}

\begin{proof}
Let $g_{D}$ be the complete K\"ahler metric given by \eqref{omega D1}.
The matrix of metric $g_D$ is denoted by $(g_{\alpha\bar \beta})$, where  $1\leq \alpha, \beta\leq m+1$.
Define $h_{j}:=\frac{\partial h}{\partial z_{j}}$, $h_{j\bar k}:=\frac{\partial^{2} h}{\partial z_{j} \partial \bar z_{k}}$, $1\leq j, k\leq m$.
By  \eqref{m1},
we get
\begin{equation}\label{equ:Hession of solution 0}
(g_{\alpha\bar \beta})=\frac{1}{(h-|v|^{2})^{2}}\left(
  \begin{array}{c|c}
  h & -h_{\bar{k}}\bar v \\
    \hline
   -h_{j}v  &  \ \ \ -(h-|v|^{2})h_{j\bar k}+h_{j}h_{\bar k} \\
  \end{array}
\right).
\end{equation}
By a directly computation, we have
\begin{equation}\label{equ:det}
\begin{aligned}
 \det(g_{\alpha\bar \beta})&=\frac{h^{m+1}}{(h-|v|^{2})^{m+2}}\det \left(\frac{h_{j}h_{\bar k}-hh_{j\bar k}}{h^{2}}\right)\\
&=\frac{h^{m+1}}{(h-|v|^{2})^{m+2}}\det \left((-\log h)_{j\bar k}\right).
\end{aligned}
\end{equation}
Inserting the determinant into
 the Ricci form
$
-\sqrt{-1}\partial \bar \partial \log \det (g_{\alpha\bar \beta}).
$
The proof is complete.
\end{proof}

\begin{proposition}\label{h sect}
Let  $\pi:(L, h)\rightarrow M$ be a positive Hermitian line bundle over
 a K\"ahler manifold $(M, g_{M})$ such that   $\omega_{M}=-\sqrt{-1}\partial \bar\partial\log h$.
 Let $(L^{*}, h^{-1})\rightarrow M$ be the dual bundle of $L$.
Consider the disc bundle
$
 D(L^{*}) := \{v \in L^{*} : |v|_{h^{-1}} < 1\},
$
where $|v|_{h^{-1}}$ denotes the norm of $v$ with respect to the metric $h^{-1}$.
Equip it with a K\"ahler metric $g_{D}$ with the K\"ahler form
$
\omega_{D}:=\pi^{*}(\omega_M)-\sqrt{-1}\partial\bar\partial \log(1-|v|_{h^{-1}}^{2}).
$
For any  fixed point $\eta_{0}\in D(L^{*})$,  there exists a local coordinate system around it
such that
the holomorphic sectional curvature of $g_{D}$ at $\eta_{0}=(z_{0}, v)$ is
$$\Theta(\eta_{0}, d\eta)=\frac{-2\left(g_{D}^{2}(\eta_{0})-\frac{g_{M}^{2}(z_{0})}{1-|v|^{2}}\left(1+\frac{1}{2}\Theta_{M}(z_{0},dz)
\right)\right)}
{g_{D}^{2}(\eta_{0})},$$
where
$\Theta_{M}(z_{0},dz)$ is the holomorphic sectional curvature of $g_{M}$, and
$$g_{M}(z_{0})=\sum \delta_{j\bar k}dz_{j}\overline{dz}_{k}, \ \
g_{D}(\eta_{0})=\frac{1}{1-|v|^{2}}\left(\frac{dv\overline{dv}}{1-|v|^{2}}+g_{M}(z_{0})\right).$$
\end{proposition}
\begin{proof}
Denote by $P_{0}\in M$ the project point of $\eta_{0}$ under the mapping $ \pi$.
 We take the geodesic
coordinate $(U, z)$ around  $P_{0}$.  The metric $g_{M}$ is denoted by  $\sum g_{j\bar k}dz_{j} d\overline{z}_{k}$.
 At point $P_{0}$, we have that $g_{j\bar k}=\delta_{j\bar k}$,  and all first derivatives of the $g_{j\bar k}$ are zero.
Let $z_{0}$ be the coordinate of  $P_{0}$,
and $\varphi$ be a K\"ahler potential of $g_{M}$ in $U$.
Let $\varphi(z,w)$ be the polarized function of $\varphi$ on $U\times \mathrm{conj}(U)$.
Then $\phi(z)=\varphi(z,\bar z)-\varphi(z,\bar z_{0})-\varphi(z_{0}, \bar z)+\varphi(z_{0},\bar z_{0})$
is an another K\"ahler potential function such that $\sqrt{-1}\partial\bar\partial\phi(z)=\sqrt{-1}\partial\bar\partial\varphi(z)=\omega_{M}$.
Recall that $-\partial\bar\partial\log h=\partial\bar\partial\phi$.
It is equivalent to
$h^{-1}|e^{f}|^{2}=e^{\phi}$
for a certain holomorphic function $f$ in $U$.
Choose a local free frame such that
$h^{-1}=e^{\phi}.$
Hence we have $$h(z_{0})=1, h_{j}(z_{0})=-e^{-\phi}\phi_{j}(z_{0})=0, h_{\bar k}(z_{0})=-e^{-\phi}\phi_{\bar k}(z_{0})=0,
h_{j\bar k}(z_{0})=-\delta_{j\bar k},$$
where $h_j=\frac{\partial h}{\partial z_{j}}$, $h_{\bar{k}}=\frac{\partial h}{\partial \bar{z}_{k}}$, $h_{j\bar k}=\frac{\partial^2 h}{\partial z_{j}\partial \bar{z}_k}$.
Let $(z, v)$ be the local coordinate of $D(L^*)$.
 Then  the matrix of the metric $g_{D}$ is
\begin{equation}\label{equ:Hession of solution 1}
T:=(g_{\alpha\bar \beta})=\frac{1}{(h-|v|^{2})^{2}}\left(
  \begin{array}{c|c}
  h & -h_{\bar{k}}\bar v \\
    \hline
   -h_{j}v  &  \ \ \ -(h-|v|^{2})h_{j\bar k}+h_{j}h_{\bar k} \\
  \end{array}
\right),
\end{equation}
where $1\leq j, k\leq m$, $1\leq \alpha, \beta\leq m+1$.
At the point $(z_{0}, v)$, we have
\begin{equation}
T(z_{0}, v)=\left(
  \begin{array}{cc}
  \frac{1}{(1-|v|^{2})^{2}} & 0 \\

  0  &  \ \ \ \frac{1}{1-|v|^{2}}I_{m} \\
  \end{array}
\right),
\end{equation}
and
\begin{equation}
T^{-1}(z_{0}, v)=\left(
  \begin{array}{cc}
 (1-|v|^{2})^{2} & 0 \\

  0  &  \ \ \ (1-|v|^{2})I_{m} \\
  \end{array}
\right).
\end{equation}
For convenience, we define
\begin{equation}
\partial T=\left(
  \begin{array}{cc}
 \partial T_{11} &  \partial T_{12} \\

   \partial T_{21}  &   \partial T_{22} \\
  \end{array}
\right).
\end{equation}
By a direct computation, we have
\begin{eqnarray*}
\partial T_{11}&=&-2(h-|v|^{2})^{-3}\partial (h-|v|^{2})\cdot h+(h-|v|^{2})^{-2}\partial h,\\
\partial T_{12}&=&(\cdots,-\partial(h-|v|^{2})^{-2}h_{\bar k}\bar v-(h-|v|^{2})^{-2}\partial h_{\bar k}\bar v,\cdots),\\
\partial T_{21}&=&(\cdots,-\partial(h-|v|^{2})^{-2}h_{j}v-(h-|v|^{2})^{-2}\partial(h_jv), \cdots)^t,
\end{eqnarray*}
where $t$ denotes the transpose of the matrix. Thus we get
$\partial T_{11}|_{z_{0}}=2(1-|v|^{2})^{-3}\bar{v} dv, \partial T_{12}|_{z_{0}}=(1-|v|^{2})^{-2}\bar v dz, \partial T_{21}|_{z_{0}}=0$.
It is easy to see that
\begin{eqnarray*}
-(h-|v|^{2})h_{j\bar k}+h_{j}h_{\bar k}&=&h(h-|v|^{2})\frac{h_{j}h_{\bar k}-h_{j\bar k}h}{h^{2}}+\frac{|v|^{2}h_{j}h_{\bar k}}{h}\\
&=&h(h-|v|^{2})g_{j\bar k}+\frac{|v|^{2}}{h}h_{j}h_{\bar k}.
\end{eqnarray*}
Let $B=(h_{j}h_{\bar k})$ and $T^{M}=(g_{j\bar k})$. Then we obtain that
\begin{eqnarray*}
\partial T_{22}&=&\partial\left(h(h-|v|^{2})^{-1}T^{M}+\frac{|v|^{2}}{h}(h-|v|^{2})^{-2}B\right)\\
&=&\partial h\cdot(h-|v|^{2})^{-1}T^{M}-h(h-|v|^{2})^{-2}\partial(h-|v|^{2})\cdot T^{M}\\
&&+h(h-|v|^{2})^{-1}\partial T^{M}+\partial[\frac{|v|^{2}}{h}(h-|v|^{2})^{-2}B].
\end{eqnarray*}
Since $\partial T^{M}|_{z_{0}}=0$, $B|_{z_{0}}=0$, $\partial B|_{z_{0}}=0$, we get
$\partial T_{22}|_{z_{0}}=(1-|v|^{2})^{-2}\bar v dv I_{m}.$ Thus, we have
\begin{equation}
\partial T_{z_{0}}=\left(
  \begin{array}{cc}
 2(1-|v|^{3})^{-3}\bar{v} dv & (1-|v|^{2})^{-2}\bar v dz \\

  0  &  \ \ \ \ (1-|v|^{2})^{-2}\bar v dv I_{m} \\
  \end{array}
\right).
\end{equation}
Notice that $\bar \partial \partial h|_{z_{0}}=-|dz|^{2}$ and $\bar \partial \partial h_{\bar k}|_{z_{0}}=0$.
For convenience, we induce some notations such as $|dv|^{2}:=dv\overline{dv}$, $|dz|^{2}:=\sum dz_{j}\overline{dz_{j}}$, $dz:=(dz_{1},\cdots,dz_{m})$. We get
\begin{eqnarray*}
\bar\partial \partial T_{11}|_{z_{0}}&=&\frac{4|v|^{2}+2}{(1-|v|^{2})^{4}}|dv|^{2}-\frac{1+|v|^{2}}{(1-|v|^{2})^{3}}\bar \partial \partial h
=\frac{4|v|^{2}+2}{(1-|v|^{2})^{4}}|dv|^{2}+\frac{1+|v|^{2}}{(1-|v|^{2})^{3}}|dz|^{2},\\
\bar\partial \partial T_{12}|_{z_{0}}&=&\frac{1+|v|^{2}}{(1-|v|^{2})^{3}}d\bar{v}dz, \ \ \ \ \ \ \ \ \ \ \ \ \ \
\bar\partial \partial T_{21}|_{z_{0}}= \frac{1+|v|^2}{(1-|v|^{2})^{3}}\overline{dz}^t dv,\\
\bar\partial \partial T_{22}|_{z_{0}}&=&\left[\frac{|v^{2}|}{(1-|v|^{2})^{2}}|dz|^{2}+\frac{1+|v|^{2}}{(1-|v|^{2})^{3}}|dv|^{2}\right]T^{M}|_{z_{0}}
+\frac{|v|^2}{(1-|v|^2)^2}\overline{dz}^tdz
+\frac{1}{1-|v|^{2}}\bar \partial \partial T^{M}|_{z_{0}}.
\end{eqnarray*}
At the point $\eta_{0}=(z_{0}, v)$, we can obtain that
\begin{align*}
&d\eta(-\bar\partial\partial T+\partial T\cdot T^{-1}\overline{\partial T}^{t})\overline{d\eta}^t|_{z_0}\\
&=\begin{pmatrix}
   dv &  dz \\
\end{pmatrix}
\begin{pmatrix}
 \frac{-2|dv|^2}{(1-|v|^2)^4}-\frac{1}{(1-|v|^2)^3}|dz|^2 & -\frac{1}{(1-|v|^{2})^{3}}d\bar v dz \\
 & \\
 -\frac{1}{(1-|v|^{2})^{3}}\overline{dz}^t dv &(-\frac{|dv|^2}{(1-|v|^{2})^{3}}-\frac{|v|^2|dz|^2}{(1-|v|^2)^2})I_{m} \\
 & -\frac{|v|^2}{(1-|v|^2)^2}\overline{dz}^tdz-\frac{1}{1-|v|^{2}}\bar \partial \partial T^{M}|_{z_{0}}\\
\end{pmatrix}
\begin{pmatrix}
  d\bar{v}\\
  \overline{dz}^t \\
\end{pmatrix}\\
&=\frac{-2|dv|^4}{(1-|v|^2)^4}-\frac{4|dv|^2|dz|^2}{(1-|v|^2)^3}-\frac{2|v|^2|dz|^4}{(1-|v|^2)^2}-\frac{1}{1-|v|^{2}}dz(\bar\partial \partial T^{M}|_{z_{0}})\overline{dz}^t\\
&=-2\left(\frac{|dv|^2}{(1-|v|^2)^2}+\frac{|dz|^2}{1-|v|^2}\right)^2+\frac{2|dz|^4}{1-|v|^2}-\frac{1}{1-|v|^{2}}dz(\bar\partial \partial T^{M}|_{z_{0}})\overline{dz}^t\\
&=-2\left[\left(\frac{|dv|^2}{(1-|v|^2)^2}+\frac{|dz|^2}{1-|v|^2}\right)^2-\frac{1}{1-|v|^2}\left(|dz|^4-\frac{1}{2}dz(\bar\partial \partial T^{M}|_{z_{0}})\overline{dz}^t\right)\right].
\end{align*}

we also notice that $g_{M}=\sum\delta_{jk}dz_{j}\otimes d\bar z_{k}$ at the point $z_0$, that is the matrix $T^{M}_{z_{0}}$ is unit matrix, therefore $|dz|^4=g^{2}_{M}(z_{0})$,
and $g_{D}(z_{0}, v)=\frac{1}{1-|v|^{2}}\left(\frac{|dv|^{2}}{1-|v|^{2}}+|dz|^{2}\right)$,
$\Theta_{M}(z_{0},dz)=\frac{dz(-\bar\partial\partial T^{M}(z_{0}))\overline{dz}^t}{|dz|^{4}}$.

We can derive that
\begin{eqnarray*}
\Theta_{D}({\eta_{0}, d\eta})&=&\displaystyle{\frac{d\eta(-\bar\partial\partial T+\partial T\cdot T^{-1}\overline{\partial T}^{t})\overline{d\eta}^t}{(d\eta T \overline{d\eta}^{t})^{2}}|_{\eta_{0}}}\\
&=&\frac{-2\left(g_{D}^{2}(\eta_{0})-\frac{g^{2}_{M}(z_{0})}{1-|v|^{2}}\left(1+\frac{1}{2}\Theta_{M}(z_{0},dz)
\right)\right)}
{g_{D}^{2}(\eta_{0})}.
\end{eqnarray*}

\end{proof}

\begin{remark}
   $\Theta(\eta_{0}, d\eta)=-2$  if and only if $\Theta_{M}(z_{0},dz)=-2$.
   They are biholomorphic to complex Hyperbolic spaces.
\end{remark}

From properties of curvature of K\"ahler manifold, the following result  can be given by (4.3), (4.4) and (4.7)  in Theorem 4.2 of \cite{Jun Nie}.
We rewrite it as follows.
Denote by $R$ the Riemannian curvature of K\"ahler manifold $(M, \omega_{M})$.
Let $X, Y$ be two vectors in complex holomorphic tangent space $T^{(1,0)}_{p}M$ and
$x=X+\overline{X}, y=Y+\overline{Y}$.
 \begin{eqnarray}\label{Relation}
  && R(x, y,y,x)=R(X+\overline{X}, Y+\overline{Y},Y+\overline{Y},X+\overline{X}) \nonumber\\
   &=& -\frac{1}{8} Q(X+Y)-\frac{1}{8}Q(X-Y)
     +\frac{3}{8}Q(X+iY) +\frac{3}{8}Q(X-iY)
    -\frac{1}{2}Q(X) -\frac{1}{2}Q(Y),
 \end{eqnarray}
 where   $Q(X)=R(X,\overline{X},X,\overline{X})$.

\begin{proposition}\label{sect}
Under the conditions in Proposition \ref{h sect},
the sectional curvature of $(D(L^{*}), g_{D})$ at $\eta_{0}=(z_{0}, v)$ is
\begin{equation}
\kappa_{D}(\mu,\nu)=-2+2\left(-\kappa_{\Omega}(x,y)+\frac{1}{2}\kappa_{M}(x,y)\right)\frac{\parallel x\wedge y \parallel^{2}_{M}}{1-|v|^{2}},
\end{equation}
where
 $g_{\Omega}$ is a complex hyperbolic metric in a local coordinate $(\Omega(z_{0}), z)$,
$\kappa_{\Omega}$ and $\kappa_{M}$ are the  sectional curvatures of $g_{\Omega}$ and $g_{M}$, respectively.
\end{proposition}
\begin{proof}
Denote by $R_{D}$ the Riemannian curvature of K\"ahler manifold $(D(L^{*}), g_{D})$.
The sectional curvature of $g_{D}$  at $\eta=(z_{0}, v)$  is
\begin{equation}
\kappa_{D}(\mu,\nu)=\frac{R_{D}(\mu,\nu,\nu,\mu)}{\parallel \mu\wedge \nu \parallel^{2}_{D}},
\end{equation}
where $\mu, \nu$ are two tangent vectors in  $T_{\eta}(D)$ at the point $\eta$,
$\parallel \mu\wedge \nu \parallel^{2}_{D}=\langle \mu,\mu\rangle_{D}\langle \nu,\nu \rangle_{D}-\langle \mu,\nu\rangle^{2}_{D}$,
$\langle \mu,\nu\rangle_{D}$ is the inner product of $\mu, \nu$ under the induced Riemannian metric $\mathrm{Re}g_{D}$.

From the definition of  sectional curvature,
we know that it is independent of the length of vectors.
Without loss of generality, we assume that $\mu, \nu$ are  orthonormal unit vectors.
Define $U=\frac{1}{2}(\mu-\sqrt{-1}J\mu), V=\frac{1}{2}(\nu-\sqrt{-1}J\nu)\in T^{(1,0)}D$.
Then $\mu=U+\overline{U}, \nu=V+\overline{V}$, and $\|U\|_{g_{D}}^{2}=\langle \mu, \mu \rangle_{D}=
\|V\|_{g_{D}}^{2}=\langle \nu, \nu \rangle_{D}=1$.
By the equality \eqref{Relation}, we have
 \begin{eqnarray}\label{R1}
 \begin{aligned}
  &  R_{D}(\mu, \nu,\nu,\mu)\\
  &=R_{D}(U+\overline{U}, V+\overline{V},V+\overline{V},U+\overline{U})\\
   &= -\frac{1}{8} Q_{D}(U+V) -\frac{1}{8}Q_{D}(U-V)
     +\frac{3}{8}Q_{D}(U+iV) +\frac{3}{8}Q_{D}(U-iV)
     -\frac{1}{2}Q_{D}(U) -\frac{1}{2}Q_{D}(V).
     \end{aligned}
 \end{eqnarray}

Notice that $T^{(1,0)}_{\eta}D=T^{(1,0)}_{z_{0}}M\oplus T^{(1,0)}_{v}\bigtriangleup=\mathrm{span}\{\frac{\partial}{\partial z_{1}},\cdots,\frac{\partial}{\partial z_{m}}, \frac{\partial}{\partial v}\}$.
Then   $U=X+X_{0}, V=Y+Y_{0}$, where $X, Y\in T^{(1,0)}_{z_{0}}M$ and $X_{0}, Y_{0}\in T^{(1,0)}_{v}\bigtriangleup$.
Consider a small neighbourhood $\Omega(z_{0})$ of $z_{0}$, equipped it with the complex hyperbolic metric $g_{\Omega}$ so that under the local coordinate
$(\Omega(z_{0}), z)$, the K\"ahler form $\omega_{\Omega}=2\sqrt{-1}\partial \bar\partial \log(1-|z|^{2})$.
Then
\begin{equation}\label{ghyp}
g_{i\bar{j}}=\frac{4(1-|z|^{2})\delta_{i\bar{j}}-z_{j}\bar z_{i}}{(1-|z|^{2})^{2}}.
\end{equation}
At the center point $z_{0}=0$, we have $g_{i\bar{j}}=\delta_{i\bar{j}}$. It is known that its holomorphic sectional curvature $\Theta_{\Omega}(X)=\frac{R_{\Omega}(X,\overline X,  X, \overline X)}{\|X\|^{4}_{g_{\Omega}}}=-1$.
It implies that $\|X\|^{4}_{g_{\Omega}(z_{0})}=-R_{\Omega}(X,\overline X,  X, \overline X)=|dz|^{4}(X,\overline{X},X,\overline{X})$.
From the above discussion, we have
 $\parallel X+Y\parallel^{4}_{g_{M}(z_{0})}=\|X+Y\|^{4}_{g_{\Omega}(z_{0})}=-R_{\Omega}(X+Y,\overline{X+Y},  X+Y, \overline{X+Y})= |dz|^{4}(X+Y,\overline{X+Y},X+Y,\overline{X+Y}),$ and
  \begin{eqnarray}
 \Theta_{M}(z_{0},dz)(X+Y)
 &=&\frac{R_{M}(X+Y,\overline{X+Y},X+Y,\overline{X+Y})}{\parallel X+Y\parallel^{4}_{g_{M}(z_{0})}}.
  \end{eqnarray}

  Now we compute the first term in the right hand of the equation \eqref{R1} by Proposition \ref{h sect}.
  \begin{eqnarray*}
 Q_{D}(U+V)&=& R_{D}(U+V,\overline{U+V},U+V,\overline{U+V})\\
 &=&\parallel U+V\parallel^{2}_{g_{D}(\eta_{0})}\Theta_{D}(\eta_{0}, d\eta)(U+V)\\
 &=&-2\left(g_{D}^{2}(\eta_{0})-\frac{|dz|^4}{1-|v|^{2}}\left(1+\frac{1}{2}\Theta_{M}(z_{0},dz)\right)\right)(U+V)\\
 &=&-2\left(4+\frac{Q_{\Omega}(X+Y)}{1-|v|^{2}}\left(1+\frac{1}{2}\frac{Q_{M}(X+Y)}{\parallel X+Y\parallel^{4}_{g_{M}(z_{0})}}\right)\right).
 \end{eqnarray*}

In  the same way, we have
  \begin{eqnarray*}
 Q_{D}(U-V)
 &=&-2\left(4+\frac{Q_{\Omega}(X-Y)}{1-|v|^{2}}\left(1+\frac{1}{2}\frac{Q_{M}(X-Y)}{\parallel X-Y\parallel^{4}_{g_{M}(z_{0})}}\right)\right),\\
 Q_{D}(U+iV)
 &=&-2\left(4+\frac{Q_{\Omega}(X+iY))}{1-|v|^{2}}\left(1+\frac{1}{2}\frac{Q_{M}(X+iY)}{\parallel X+iY\parallel^{4}_{g_{M}(z_{0})}}\right)\right),\\
 Q_{D}(U-iV)
 &=&-2\left(4+\frac{Q_{\Omega}(X-iY)}{1-|v|^{2}}\left(1+\frac{1}{2}\frac{Q_{M}(X-iY)}{\parallel X-iY\parallel^{4}_{g_{M}(z_{0})}}\right)\right),\\
 Q_{D}(U)
 &=&-2\left(1+\frac{Q_{\Omega}(X)}{1-|v|^{2}}\left(1+\frac{1}{2}\frac{Q_{M}(X)}{\parallel X\parallel^{4}_{g_{M}(z_{0})}}\right)\right),\\
 Q_{D}(V)
 &=&-2\left(1+\frac{Q_{\Omega}(Y)}{1-|v|^{2}}\left(1+\frac{1}{2}\frac{Q_{M}(Y)}{\parallel Y\parallel^{4}_{g_{M}(z_{0})}}\right)\right).
 \end{eqnarray*}
\end{proof}

Let $x=X+\overline{X}, y=Y+\overline{Y}\in T_{z_{0}}(M)$.
Insert the equations above into \eqref{R1}, it turns to be
$$R_{D}(\mu,\nu,\nu,\mu)=-2\left[1+\frac{R_{\Omega}(x,y,y,x)}{1-|v|^{2}}-\frac{1}{2}\frac{R_{M}(x,y,y,x)}{1-|v|^{2}}\right] \ \text{at} \ \eta_{0}=(z_{0}, v) \in D(L^{*}).$$

If $x=0$ or $y=0$, then $R_{D}(\mu,\nu,\nu,\mu)=-2$.
In the following, we assume that $x$ and $y$ are non-zero vectors.
 Notice that $\langle \mu,\nu\rangle_{D}=0$, $\langle \mu,\mu\rangle_{D}=\|U\|_{g_{D}}^{2}=1$, $\langle \nu,\nu\rangle_{D}=\| V\|_{g_{D}}^{2}=1$, we have  $\parallel \mu\wedge \nu \parallel^{2}_{D}=\langle \mu,\mu\rangle_{D}\langle \nu,\nu \rangle_{D}-\langle \mu,\nu\rangle^{2}_{D}=1$.
 At point $z_{0}$, we have $\parallel x\wedge y \parallel^{2}_{\Omega}=\parallel x\wedge y \parallel^{2}_{M}=\langle x,x\rangle_{M}\langle y,y\rangle_{M}-\langle x,y\rangle^{2}_{M}.$
The sectional curvature  of $(D, g_{D})$  at $\eta=(z_{0}, v)$ is
\begin{eqnarray}\label{kappa}
\kappa_{D}(\mu,\nu)&=&R_{D}(\mu,\nu,\nu,\mu)=-2\left[1+\frac{\kappa_{\Omega}(x,y)\parallel x\wedge y \parallel^{2}_{\Omega}}{1-|v|^{2}}-\frac{1}{2}\frac{\kappa_{M}(x,y)\parallel x\wedge y \parallel^{2}_{M}}{1-|v|^{2}}\right]\\
&=&-2+2\left(-\kappa_{\Omega}(x,y)+\frac{1}{2}\kappa_{M}(x,y)\right)\frac{\parallel x\wedge y \parallel^{2}_{M}}{1-|v|^{2}}
\end{eqnarray}
where $\kappa_{\Omega}$ and $\kappa_{M}$ are the sectional curvatures of $g_{\Omega}$ in \eqref{ghyp} and $g_{M}$ respectively.


\section{Negatively pinched properties}
In this section, we will study the $\delta$-pinched properties of the disc bundles by  estimating the (holomorphic) sectional curvature and Ricci curvature.

\begin{theorem}\label{Thm1}
The K\"ahler manifold  $(D(L^{*}), g_{D})$ is negatively holomorphically pinched  if and only if $(M, g_{M})$  is so.
\end{theorem}
\begin{proof}
By \eqref{m1}, we know
$(M, g_{M})$  is the  K\"ahler submanifold of $D(L^{*})$.
Thus we have
$$\Theta_{M}(x, Jx)=\Theta_{D}(x, Jx)-\frac{2(g_{M}(B(x,x),B(x,x))}{(g_{M}(x,x))^{2}} \ \ \text{for} \ x\in T(M), x\neq0,$$
where  $B(x,x)$ is the second fundamental form of $(M, g_{M})$ in $D$.
The necessity is obvious. It therefore suffices to prove
 the sufficiency.

Assume that  $C_{1}\leq\Theta_{M}(X)\leq C_{2}<0$ for $\ X\in T^{(1,0)}(M), X\neq0$,
and $\parallel U\parallel_{g_{D}(\eta_{0})}=1$.
Since
 $\| U\|^{2}_{g_{D}(\eta_{0})}=\frac{1}{(1-|v|^{2})^{2}}|X_{0}|^{2}+\frac{1}{(1-|v|^{2})} \|X\|_{g_{M}(z_{0})}^{2},$
we get
\begin{equation}\label{s}
0\leq\frac{g^{2}_{M}(z_{0})(X,X)}{g_{D}^{2}(\eta_{0})(U,U)}\leq 1-|v|^{2}.
\end{equation}
The left-hand equality holds for $U=X_{0}$, while the right-hand equality holds for $U=X$.
By Proposition \ref{h sect}, we have
$$\Theta_{D}(\eta_{0}, d\eta)(U)=-2+\frac{1}{1-|v|^{2}}\frac{g^{2}_{M}(z_{0})(X,X)}{g_{D}^{2}(\eta_{0})(U,U)}\left(2+\Theta_{M}(z_{0},dz)(X)
\right).$$
A simple estimation shows that
\begin{equation}\label{11}
\Theta_{D}(\eta_{0}, d\eta)(U)\leq\left\{
  \begin{array}{ll}
    C_{2}, & \text{if} \ C_{2}\geq-2 \hbox{;} \\
    -2, &  \text{if} \ C_{2}\leq-2\hbox{;}
  \end{array}
\right.
\end{equation}
and
\begin{equation}\label{22}
\Theta_{D}(\eta_{0}, d\eta)(U)\geq\left\{
  \begin{array}{ll}
    -2, & \text{if} \ C_{1}\geq-2 \hbox{;} \\
    C_{1}, &  \text{if} \ C_{1}\leq-2\hbox{.}
  \end{array}
\right.
\end{equation}
By the arbitrariness of $\eta_{0}$, we obtain that
$\min \{-2, C_{1}\}\leq \Theta_{D}\leq \max \{-2, C_{2}\}$ on $D(L^{*})$.
 This estimation is sharp since \eqref{s} is sharp.
\end{proof}

The following result shows a comparison on the pinched constants between the disc bundle  and its base space.
\begin{corollary}
 If $(M, g_{M})$  is negatively $\delta$-holomorphically  pinched, then
   $(D(L^{*}), g_{D})$   is negatively $\delta'$-holomorphically  pinched, where $\delta'\geq \delta$ when $A\geq2$, and $\delta'< \delta$ when $0< A <2$.
\end{corollary}
\begin{proof}
Take $C_{1}=-A, C_{2}=-\delta A$ in \eqref{11} and \eqref{22}.

Case 1.
If $A\geq2$, $0<\delta\leq \frac{2}{A}$, then $-A\leq\Theta_{D}\leq -\delta A$, i.e.,
 $\delta'=\delta$;

Case 2.
If $A\geq2$, $\frac{2}{A}\leq\delta\leq1$, then $-A\leq\Theta_{D}\leq-2$, i.e.,
 $\delta'=\frac{2}{A}\geq\delta$;

Case 3.
If $A<2$,  then $-2\leq\Theta_{D}\leq -\delta A$, i.e., $\delta'=\delta\frac{A}{2}<\delta<1$.
\end{proof}

A directly corollary can be obtained by using Wu and Yau's result \cite{Wu2020} and Theorem \ref{Thm1}.
\begin{corollary}
The disc bundle $(D(L^{*}), g_{D})$ over a complete negatively holomorphic pinched K\"ahler manifold  $(M, g_{M})$ has
 a unique  complete K\"ahler-Einstein metric.
Moreover, the Kobayashi metric and K\"ahler-Einstein metric are equivalent.
\end{corollary}

\begin{theorem}\label{c1}
The K\"ahler manifold  $(D(L^{*}), g_{D})$ is negatively pinched  if and only if $(M, g_{M})$  is so.
\end{theorem}
\begin{proof}
Let  $\kappa_{D}$ be the sectional curvatures of $(D, g_{D})$.
At $\eta_{0}=(z_{0}, v)$, we have
\begin{equation}
\kappa_{D}(\mu,\nu)=-2+2\left(-\kappa_{\Omega}(x,y)+\frac{1}{2}\kappa_{M}(x,y)\right)\frac{\parallel x\wedge y \parallel^{2}_{M}}{1-|v|^{2}},
\end{equation}
where $\kappa_{\Omega}$ and $\kappa_{M}$ are the sectional curvatures of $g_{\Omega}$ in \eqref{ghyp} and $g_{M}$, respectively.

Let $\Pi$ be a plane in
$T_{z_{0}}(M)$, i.e., a real $2$-dimensional subspace of $T_{z_{0}}(M)$. Let $x$ and
$y$ be an orthonormal basis. Define the
angle $\alpha(\Pi)$ between $\Pi$ and $J(\Pi)$ by $\cos^{2}\alpha(\Pi)=\|g(x, Jy)\|$.
It is well known that
the sectional curvature of a space of constant holomorphic sectional curvature
$c$ is given by $\frac{c}{4}(1 + 3 \cos^{2}\alpha(\Pi))$ (See Note 23 in \cite{Kobayashi1963} or Proposition 3.6.1 in \cite{Kim2011}).
Hence, we have
$$-1\leq\kappa_{\Omega}(x,y)\leq-\frac{1}{4}.$$

On the one hand,
$1=\| U\|^{2}_{g_{D}}=\frac{1}{(1-|v|^{2})^{2}}|X_{0}|^{2}+\frac{1}{(1-|v|^{2})} \|X\|_{g_{M}}^{2}$
in terms of the formula of  $g_{D}$ at $\eta_{0}$ in Proposition \ref{h sect};
On the other hand $\langle x,  x\rangle^{2}_{M}=\|X\|_{g_{M}}^{2}.$ We have  $\langle x,  x\rangle_{M}\leq 1-|v|^{2}.$
In the same way, we have $\langle y,  y\rangle_{M}\leq 1-|v|^{2}.$
Let $\theta$ be the angle between $x$ and $y$.
Then
\begin{equation}\label{equ a}
0\leq\frac{\parallel x\wedge y \parallel^{2}_{M}}{1-|v|^{2}}=\frac{1}{1-|v|^{2}}\|x\|^{2}_{M}\|y\|^{2}_{M}|\sin\theta|^{2}\leq(1-|v|^{2}).
\end{equation}
Let $x_{0}=X_{0}+\overline{X}_{0}, y_{0}=Y_{0}+\overline{Y}_{0}\in T_{v}(\Delta)$.
Then the tangent vectors $\mu$ and $\nu$ can be expressed by $x_{0}+x$ and $y_{0}+y$.
The equality on the left side holds when we choose $u$ such that  $x$ vanishes, while the right one holds when we choose $\mu$ and
 $\nu$ such that $x_{0}$, $y_{0}$ vanish.
More precisely, the last case comes form the assumption that  $\mu, \nu$ are  orthonormal unit vectors.
If $x_{0}$, $y_{0}$ vanish in the tangent vectors $\mu$ and
 $\nu$, then $\langle x,  x\rangle_{M}=\langle y,  y\rangle_{M}= 1-|v|^{2}$ and $\theta=\frac{\pi}{2}$.

Assume that there are two negative constants $c_{1}$ and $c_{2}$ such that  $c_{1}\leq \kappa_{M}(z_{0}, dz)\leq c_{2} $.
By a simple estimation, we have
\begin{equation}\label{111}
\kappa_{D}(\mu,\nu)\leq\left\{
  \begin{array}{ll}
    c_{2}, & \text{if} \ c_{2}\geq-2 \hbox{;} \\
    -2, &  \text{if} \ c_{2}\leq-2\hbox{;}
  \end{array}
\right. \ \
 \ \
\kappa_{D}(\mu,\nu)\geq\left\{
  \begin{array}{ll}
    -2, & \text{if} \ c_{1}\geq-\frac{1}{2} \hbox{;} \\
    -\frac{3}{2}+c_{1}, &  \text{if} \ c_{1}\leq-\frac{1}{2}\hbox{.}
  \end{array}
\right.
\end{equation}

Due to the arbitrariness of $\eta_{0}$ and the fact that the space is spanned by $ \{\mu, \nu\}$,
it implies that $$\min \{-2, -\frac{3}{2}+c_{1}\}\leq\kappa_{D}\leq \max \{-2, c_{2}\}.$$
This estimation is sharp since \eqref{equ a} is sharp.
We have completed the proof.
\end{proof}

\begin{remark}
$\kappa_{D}=-2$ if and only if $\kappa_{M}=-2\kappa_{\Omega}$.
They are real Hyperbolic spaces.
\end{remark}

 We now compare the pinched constants.
\begin{corollary}
  Let $(M, g_{M})$  be  $\delta$-negatively pinched,
  and  $(D(L), g_{D})$   be  $\delta'$- negatively pinched. Then we have that $\delta'\geq \frac{1}{4}\delta$ when $A\geq\frac{1}{2}$, and $\delta'< \frac{1}{4}\delta$ when $0< A <\frac{1}{2}$.
\end{corollary}
\begin{proof}
Take $C_{1}=-A, C_{2}=-\delta A$ in \eqref{111}, we arrive at the following cases.

Case 1.
If $A\geq\frac{1}{2}$, $0<\delta\leq \min\{1,\frac{2}{A}\}$, then $-\frac{3}{2}-A\leq\kappa\leq -\delta A$, i.e.,
 $\delta'=\frac{\delta A}{A+\frac{3}{2}}\geq \frac{1}{4}\delta$;

Case 2.
If $A\geq\frac{1}{2}$, $\min\{1,\frac{2}{A}\}\leq\delta\leq1$, then $-\frac{3}{2}-A\leq\kappa\leq-2$, i.e.,
 $\delta'=\frac{2}{A+\frac{3}{2}}\geq\delta$;

Case 3.
If $A<\frac{1}{2}$, $0<\delta\leq1$, then $-2\leq\kappa\leq -\delta A$, i.e., $\delta'=\delta\frac{A}{2}<\frac{1}{4}\delta<1$.

\end{proof}

By Wu and Yau's result and Lemma \ref{c1}, we have
\begin{corollary}
If $D(L^*)$ is  simple-connected, and the sectional curvature of $(M, g_{M})$ is pinched between by two  negative constants, then there exists a complete  Bergman metric on $D(L^*)$.
Moreover, the Bergman metric, the K\"aher-Einstein metric, Kobayashi metric and the background metric are all equivalent.
\end{corollary}

\begin{theorem}
Let  $\pi:(L, h)\rightarrow M$ be a positive Hermitian line bundle over
 a K\"ahler manifold $(M, g_{M})$  satisfying  $\omega_{M}=-\sqrt{-1}\partial \bar\partial\log h$.
 Let $(L^{*}, h^{-1})\rightarrow M$ be the dual bundle of $L$.
Consider the unit disc bundle
$
 D(L^{*}) := \{v \in L^{*} : |v|_{h^{-1}} < 1\},
$
where $|v|_{h^{-1}}$ denotes the norm of $v$ with respect to the metric $h^{-1}$.
Equip it with a K\"ahler metric $g_{D}$ whose  K\"ahler form
$
\omega_{D}:=\pi^{*}(\omega_M)-\sqrt{-1}\partial\bar\partial \log(1-|v|_{h^{-1}}^{2}).
$
If the Ricci curvature of $g_{M}$ is less than $1$,
then
 the disc bundle  $D(L^{*})$   over compact K\"ahler manifold  $(M, g_{M})$ has
 a unique  complete K\"ahler-Einstein metric.
\end{theorem}
\begin{proof}
Define $h_{r}=rh$ for a fixed $r\in \mathbb{R}^{+}$. Then
 $\pi:(L, h_{r})\rightarrow M$ is a positive Hermitian line bundle over
 $(M, g_{M})$  satisfying  $\omega_{M}=-\sqrt{-1}\partial \bar\partial\log h_{r}$.
 Consider the unit disc bundle
$
 D_{r}(L^{*}) := \{v \in L^{*} : |v|_{h_{r}^{-1}} < 1\},
$
where $|v|_{h_{r}^{-1}}$ denotes the norm of $v$ with respect to the metric $h_{r}^{-1}$.
Equip it with a K\"ahler metric $g_{D_{r}}$ with  K\"ahler form
$
\omega_{D_{r}}:=\pi^{*}(\omega_M)-\sqrt{-1}\partial\bar\partial \log(1-|v|_{h_{r}^{-1}}^{2}).
$
Since $|v|_{h_{r}^{-1}}=\frac{1}{r}|v|_{h^{-1}}$,  we have
$
 D_{r}(L^{*}) = \{v \in L^{*} : |v|_{h^{-1}} < r\}.
$
Suppose that $(M, g_{M})$ is a compact K\"ahler manifold.
For $r>1$, we have  $ D(L^{*})\subset\subset D_{r}(L^{*})$.
Then $D(L^{*})$ is strictly pseudoconvex domain in $D_{r}(L^{*})$.

By Lemma \ref{ric}, we know
\begin{eqnarray*}
\frac{\mathrm{Ric}(g_{D_{r}})}{g_{D_{r}}}&=&-(m+2)+\frac{g_{M}}{g_{D_{r}}}\left((m+1)+\frac{\mathrm{Ric}(g_{M})}{g_{M}}\right)\\
&<& \max\left\{-(m+2), -1+\frac{\mathrm{Ric}(g_{M})}{g_{M}}\right\}.
\end{eqnarray*}
Hence, $D_{r}(L^{*})$ admits a K\"ahler metric $g_{D_{r}}$  such that its Ricci
curvature is negative on $\overline{D(L^{*})}$ if the Ricci curvature of $g_{M}$ is less than $1$.
 By Cheng and Yau's Corollary 4.7 in \cite{Cheng1980}, we have now established the proof.
\end{proof}
\section{A family of ball bundles}
Let $(L, h)$ be a positive line bundle over the complex manifold $M$. For any fix $k\in \mathbb{Z}^{+}$, set
$$(E_{k}, H_{k}) = (L, h)\oplus \cdots\oplus (L, h).$$
There are $k$ copies of $(L, h)$ on the right hand side.
The dual vector bundle is
$$(E_{k}^{*}, H_{k}^{*}) = (L^{*}, h^{-1})\oplus \cdots\oplus (L^{*}, h^{-1}).$$
The ball bundle is defined by
 \begin{equation}\label{Bk}
 B(E_{k}^{*}) := \{v \in E_{k}^{*} : |v|_{H_{k}^{*}}^{2} < 1\}.
 \end{equation}

Define an $(1,1)$-form on $E_{k}^{*}$ by
\begin{equation}\label{omega Dk}
\omega_{B(E_{k}^{*})}:=\pi^{*}(\omega_M)-\sqrt{-1}\partial\bar\partial \log(1-|v|_{H_{k}^{*}}^{2}),
\end{equation}
where $\omega_M=-\sqrt{-1}\partial\overline{\partial}\log h$,
($\omega_{M_{k}}:=-\sqrt{-1}\partial\overline{\partial}\log \det(H_{k})=-k\sqrt{-1}\partial\overline{\partial}\log h=k\omega_M$).

Notice that $E_{k}^{*}$ is a line bundle over  $E_{k-1}^{*}$.
We restrict $E_{k}^{*}$  on $B(E_{k-1}^{*})$, and denote it by $\pi_{k}:L_{k}^{*}\rightarrow B(E_{k-1}^{*})$.
Since $E_{k}^{*}=E_{k-1}^{*}\oplus L^{*}$,  for $v\in E_{k}$, we have
$v=v'\oplus v_k$, where $v'\in E_{k-1}^{*}$ and $v_{k}$ is the 1 dimensional fiber.
Define $\widetilde{h}_{k}=h^{-1}(1-|v'|^{2}_{H_{k-1}^{*}})^{-1}$ as a metric on the line bundle $L^*_k$.
The curvature $-\sqrt{-1}\partial\overline{\partial}\log \widetilde{h}_{k}=-\omega_{B(E_{k-1}^{*})}$.
The Hermite line bundle $(L_{k}^{*}, \widetilde{h}_{k})$ over  $B(E_{k-1}^{*})$ is negative and admits a
 disc bundle
 \begin{equation}\label{}
D(L_{k}^{*})=  \{v_{k} \in L_{k}^{*}: |v_{k}|_{\widetilde{h}_{k}}^{2} < 1, v' \in B(E_{k-1}^{*})\},
 \end{equation}
where $B(E_{0}^{*})$ denotes $M$. Then we have
$D(L_{k}^{*})= B(E_{k}^{*})$.
Define  the $(1,1)$-form  $$\omega_{D(L_{k}^{*})}=\pi_{k}^{*}(\omega_{B(E_{k-1}^{*})})-\sqrt{-1}\partial\bar\partial \log(1-|v_{k}|_{\widetilde{h}_{k}}^{2}),
$$
Then we have $\omega_{D(L_{k}^{*})}=\omega_{B(E_{k}^{*})}.$
This implies that $B(E_{j}^{*})$ can be seen as a unit disc bundle over $B(E_{j-1}^{*})$ for $1\leq j\leq k$.

By  Theorem \ref{Thm1} and Theorem \ref{c1}, we can reduce the research on the negatively pinched properties of the ball bundle  $B(E_{k}^{*})$ to that of the disc bundle $D(L^{*})$ over $M$.
Moreover, if $(M, g_{M})$ is complete  negatively (holomorphically) pinched, so does the ball bundle $(B(E_{k}^{*}), g_{k})$.
 Hence, we have the following results.

 \begin{theorem}
 If $(M, g_{M})$ is a complete K\"ahler metric with negatively pinched holomorphic sectional curvature,
then there exists a unique  complete K\"ahler-Einstein metric on $B(E^{*}_{k})$.
Moreover, the Kobayashi metric and K\"ahler-Einstein metric is equivalent.
 \end{theorem}

 \begin{theorem}
 If $(M, g_{M})$ is a simple-connected complete K\"ahler manifold with negatively pinched sectional curvature,
 then there exists a complete  Bergman metric on $B(E^{*}_{k})$.
Moreover, the Bergman metric, the K\"aher-Einstein metric, Kobayashi metric and the background metric are all equivalent.
 \end{theorem}



\end{document}